\documentclass{amsart}
\usepackage{graphicx}
\usepackage{cite}
\usepackage{amssymb}
\newtheorem{thm}{Theorem}[section]

\newtheorem{lem}[thm]{Lemma}

\theoremstyle{definition}

\theoremstyle{remark}
\newtheorem{rem}[thm]{Remark}
\numberwithin{equation}{section}

\begin{document}

\title[$L^p$ polyharmonic Robin problems on Lipschitz domains]{$L^p$ polyharmonic Robin problems on Lipschitz domains}%
\author{Weifeng Li, Pei Dang, Zhihua Du, Guoan Guo and Yumei Li}%
\address{College of Mathematics and Physics, China University of Geosciences, Wuhan 430074, China}%
\email{liwf@cug.edu.cn}%
\address{Faculty of Information Technology, Macau University of Science and Technology,
Macao}%
\email{pdang@must.edu.mo}%
\address{Department of Mathematics, Jinan University, Guangzhou 510632, China, and Faculty of Information Technology, Macau University of Science and Technology,
Macao}%
\email{tzhdu@jnu.edu.cn and zhdu@must.edu.mo}%
\address{Colledge of Mathematics, Nanjing University of Posts and Telecommunications, Nanjing 210023, China}%
\email{guoguoan@njupt.edu.cn}%
\address{Department of Mathematics, Jinan University, Guangzhou 510632, China}%
\email{lym1049@163.com}%


\subjclass{31B10, 31B30, 35J40}%
\keywords{polyharmonic equations, Robin problems, polyharmonic fundamental solutions, multi-layer potentials}%

\begin{abstract}
In this paper, we study a class of boundary value problems (BVPs) with Robin conditions in some $L^p$ spaces for polyharmonic equation on Lipschitz domains. Utilizing polyharmonic fundamental solutions, these Robin BVPs are solved by the method of layer potentials. The crucial ingedients of our approach are the classical single layer potential and its higher order analog (which are called multi-layer $S$-potentials), and the main results generalize ones of second order (Laplacian) case to higher order (polyharmonic) case.
\end{abstract}
\maketitle
\section{Introduction}

In classical potential theory, it is well known that three classes of boundary value problems for Laplacian are important. They are Dirichlet, Neumann and Robin problems stated respectively as follows
\begin{equation*}{(D)}\hspace{5mm}\begin{cases}\Delta u=0\,\,in\,\, D,\\
 u=f\,\,on\,\,\partial D,
  \end{cases}
  \end{equation*}

\begin{equation*}{(N)}\hspace{5mm}\begin{cases}\Delta u=0\,\,in\,\, D,\\
 \frac{\partial u}{\partial N}=g\,\,on\,\,\partial D,
  \end{cases}
  \end{equation*}

 \begin{equation*}{(R)}\hspace{5mm}\begin{cases}\Delta u=0\,\,in\,\, D,\\
 \frac{\partial u}{\partial N}+bu=h\,\,on\,\,\partial D,
  \end{cases}
  \end{equation*}
 where $D$ is a domain in $\mathbb{R}^{n}$ with boundary $\partial D$, $\Delta$ is the classical Laplacian for $\mathbb{R}^{n}$ (that is, $\Delta=\sum_{j=1}^{n}\frac{\partial^{2}}{\partial x_{j}^{2}}$), $f,g,b,h$ are given functions in some certain function spaces defined on $\partial D$ such as $C(\partial D)$, $L^{p}(\partial D)$, etc., $\frac{\partial u}{\partial N}$ denotes the outward normal derivative of $u$, and $b$ is ordinarily called Robin coefficient.

Historically, there were a lot of investigations on these second order elliptic BVPs for different typed boundary data in various domains in $\mathbb{R}^{n}$ or some certain manifolds \cite{hel,f1,st1}. Many theories and techniques were invented in order to attack these boundary value problems such as Green functions, subharmonic functions, potential theory, calculus of variations, Fredholm theory of integral equations and so on. Due to the techniques of modern harmonic analysis, in recent thirty years, there were a great deal of activities for studying these fundamental harmonic BVPs and related problems with minimal smoothness boundary data (such as, $L^{p}$, $H^{p}$, BMO etc.) in non-smooth domains (for instance, $C^{1}$, Lipschitz domains etc.) in $\mathbb{R}^{n}$ or some certain Riemann manifolds \cite{dk,dk1,fjv,d,v0,k}.

Although the progress of harmonic boundary value problems is spectacular, the studies of polyharmonic boundary value problems are far from reaching a satisfactory level of good understanding. There are many differences about boundary value problems for Laplace (more generally, second order elliptic partial differential) equation and plolyharmonic (more generally, higher order elliptic partial differential) equation. For instance, it is well known that maximum principle is a crucial tool to assure the existence and uniqueness of second order BVPs, while there is not a good maximum principle in the general case of higher order partial differential equations \cite{ag,mp,pv2}. So many methods developed for second order case will not be invalid in higher order case. To attach higher order elliptic BVPs, some new ideas and techniques must be introduced. In addition, the neccesary boundary conditions to obtain a well-posed theory of BVPs are determined in second order case more easily than higher order case since the optional boundary conditions of the latter are more plentiful than the ones of the former. Absence of positivity, abundance of boundary conditions and many other differences result in more difficulties in the investigation of polyharmonic (generally, higher order elliptic) BVPs than harmonic (in general, second order elliptic) BVPs. Just as the role of Laplace equation taking on the study of second order elliptic PDEs, polyharmonic equation is vital to the study for higher order elliptic PDEs as a model equation at least in the case of higher even order.

In 2008, Begehr, Du and Wang studied a class of polyharmonic Dirichlet boundary value problems (some times, also called Requier problems \cite{r}) with H\"{o}lder continuous data in the unit disc as follows
\begin{equation*}{(PHD)}\hspace{5mm}\begin{cases}\Delta^{m} u=0\,\,in\,\, \mathbb{D},\\
 \Delta^{j} u=f_{j}\,\,on\,\,\partial \mathbb{D},
  \end{cases}
  \end{equation*}
where $0\leq j<m$, $\mathbb{D}$ is the unit disc in the complex plane, its boundary $\partial \mathbb{D}$ is the unit circle, and $f_{j}$ are H\"older continuous functions defined on $\partial \mathbb{D}$. There they found that the solutions can be represented by some integrals in terms of some kernel functions and  tried to utilize complex integrals method to explicitly formulate the kernels. However their techniques to obtain kernel functions are complicated and indeed valid only for some small $m$. In \cite{dgw}, Du, Guo and Wang introduced some new ideas and got explicit expressions of the kernel functions in a unified way. From then on, these kernel functions are called higher order Poisson kernels because they are higher order analog of the classical Poisson kernel. Then Du and his collaborators also obtained higher order Poisson kernels of some regular domains such as the unit ball, the upper half plane and space, and applied higher order Poisson kernels to solve the corresponding polyharmonic Dirichlet problems (that is, the above PHD problem) in these regular domains (\cite{dkw,dqw1,dqw2,dqw3}). Furthermore, in \cite{du}, excepting the above polyharmonic Dirichlet problems with $L^{p}$ boundary data, Du also studied the polyharmonic Neumann and regularity problems on bounded Lipschitz domains and Lipschitz graph domains as follows:
\begin{equation*}{(PHN)}\hspace{5mm}\begin{cases}\Delta^{m} u=0\,\,in\,\, D,\\
 \frac{\partial}{\partial N}\Delta^{j} u=g_{j}\,\,on\,\,\partial D,
  \end{cases}
  \end{equation*}
where $0\leq j<m$, $D$ is a bounded Lipschitz domain or Lipschitz graph domain in $\mathbb{R}^{n}$, $g_{j}\in L^{p}(\partial D)$ for some $1<p<\infty$, $0\leq j<m$;
\begin{equation*}{(PHRe)}\hspace{5mm}\begin{cases}\Delta^{m} u=0\,\,in\,\, D,\\
 \Delta^{j} u=l_{j}\,\,on\,\,\partial D,
  \end{cases}
  \end{equation*}
where $0\leq j<m$, $D$ is a bounded Lipschitz domain and Lipschitz graph domain in $\mathbb{R}^{n}$, $l_{j}\in L_{1}^{p}(\partial D)$ for some $1<p<\infty$, $0\leq j<m$.

The key techniques to solve all these polyharmonic BVPs are higher order Poisson kernels, polyharmonic fundamental solutions and some properties (mainly, the invertible properties) of the classical (single and double) layer potentials. The invertible properties of traces (or boundary values) of the classical single and double layer potentials were developed by Fabes, Verchota, Dahlberg, Kenig and others in 1980s. In fact, Dahlberg, Kenig and Vertocha and their collaborators solved  Dirichlet, Neumann, and regularity problems and some related problems for Laplace equation in $L^{p}$ on Lipschitz domains by the method of layer potentials \cite{dk,dk1,fjv,v0}.

In 2004, Lanzani and Shen studied the aforementioned Robin problem (\emph{R}) for the Laplacian in $L^{p}$ on Lipschitz domains. They established the invertibility of a crucial operator $\mathcal{T}$ which is the Robin boundary value of single layer potential under certain conditions that the Robin coefficient $b$ satisfies. Then they obtained the unique solution of Robin problem with $L^{p}$ data on bounded Lipschitz domains and $C^{1}$ domains in term of the inverse of the operator $\mathcal{T}$ and the single layer potential (The details can be seen in \cite{ls} and the $L^{2}$ case in the setting of Clifford analysis can be also seen in \cite{l} due to Lanzani).

By applying some modern harmonic analysis techniques and the method of layer potentials, more earlier, Cohen, Gosselin, Dahlberg, Kenig, Verchota and Pipher and many others had already studied some boundary value problems for biharmonic and polyharmonic (even higher order elliptic) equations in non-smooth domains (mainly, Lipschitz and $C^{1}$ domains)\cite{cg,dkv,pv1,pv2,pv3,v1,v2,v}. For example, in \cite{dkv}, Dahlberg, Kenig and Verchota investigated the following biharmonic BVP (which is called Dirichlet problem by them; in fact, the boundary data of this BVP is a kind of mixed boundary value which is coupled by Dirichlet and Neumann boundary values) on Lipschitz domains:
\begin{equation*}\begin{cases}\Delta^{2} u=0\,\,in\,\, D,\\
  u=f\,\,on\,\,\partial D,\\
 \frac{\partial}{\partial N} u=g\,\,on\,\,\partial D,\\
 M(\nabla u)\in L^{p}(\partial D),
  \end{cases}
  \end{equation*}
where $f\in L^{p}_{1}(\partial D)$, $g\in L^{p}(\partial D)$ for some $p\in (2-\epsilon, 2+\epsilon)$, $\epsilon=\epsilon(D)$ is a constant depending only on the Lipschitz domain $D$, $M(\nabla u)$ is the non-tangential maximal function of $\nabla u$ which is defined in next section.

Taking the types of boundary values into consideration, whether Begehr, Du and Wang's or Dahlberg, Kenig, Verchota and Pipher's results are only one-case studies. They can not be included and inferred each other (So it does for almost all results on polyharmonic BVPs presented in the literature). In fact, they are respectively independent and non-intersecting in some sense.

In present paper, we will also do a one-case study for BVPs for polyharmonic equations in $L^{p}$ on Lipschitz domains. More precisely, we will solve the following polyharmonic Robin problems:
\begin{equation*}{(PHR)}\hspace{5mm}\begin{cases}\Delta^{m} u=0\,\,in\,\, D,\\
 \frac{\partial }{\partial N}\Delta^{j}u+b_{j}\Delta^{j}u=h_{j}\,\,on\,\,\partial D,
  \end{cases}
  \end{equation*}
where $0\leq j<m$, $D$ is a bounded Lipschitz domain in $\mathbb{R}^{n}$, $h_{j}\in L^{p}(\partial D)$ for some $1<p<\infty$, Robin coefficients $b_{j}\in L^{s}(\partial D)$ for some $1<s<\infty$, $b_{j}\geq0$ and $b_{j}\not\equiv 0$ on $\partial D$, $0\leq j<m$.

Our approch is based on the second order results in \cite{ls} due to Lanzani and Shen and the theory of multi-layer $\mathcal{S}$-potentials constructed in terms of polyharmonic fundamental solutions in \cite{du}.

Ordinarily, a trivial idea to solve the above polyharmonic BVP is iteration. That is, in order to solve the polyharmonic BVPs, the original BVPs will be decomposed into  a series of homogeneous and inhomogeneous harmonic BVPs by constantly iterating the corresponding harmonic solutions. For instance, when $b_{j}=0$, $0\leq j\leq m-1$ and $m=2$, the above PHR problem is degenerated to be a biharmonic Neumann problem as follows:
\begin{equation*}\begin{cases}\Delta^{2} u=0\,\,in\,\, D,\\
 \frac{\partial }{\partial N}\Delta u=h_{1}\,\,on\,\,\partial D,\\
  \frac{\partial }{\partial N}u=h_{0}\,\,on\,\,\partial D.
  \end{cases}
  \end{equation*}
By the method of iteration, the following two harmonic BVPs should be successively solved. More precisely,
\begin{equation*}\begin{cases}\Delta w=0\,\,in\,\, D,\\
 \frac{\partial }{\partial N}w=h_{1}\,\,on\,\,\partial D
  \end{cases}
  \end{equation*}
and
\begin{equation*}\begin{cases}\Delta u=w\,\,in\,\, D,\\
 \frac{\partial }{\partial N}u=h_{0}\,\,on\,\,\partial D.
  \end{cases}
  \end{equation*}

Although these harmonic BVPs with $L^{p}$ data could be solved for some appropriate $p\in (1,\infty)$ by Jerison-Kenig's theory (for details, see \cite{jk1,jk2}), the solution of the original polyharmonic BVPs will be expressed in a complex form since some iterations of solution appear in the solving process. The innovation of our approach in present paper lies in the introduction of multi-layer $\mathcal{S}$-potentials (which are higher order analog of the classical singular layer potential) by studying the polyharmonic fundamental solutions. Then the aforementioned polyharmonic Robin problem will be neatly solved, and its solution can be explicitly expressed in terms of some multi-layer $\mathcal{S}$-potentials.

\section{Preliminaries}
In this section, we give some notations and recall some basic results on multi-layer $\mathcal{S}$-potentials and harmonic Robin problem which will be applied throughout the paper.

Let $D$ be a bounded Lipschitz domain in $\mathbb{R}^{n}$, $n\geq 3$. That is, $D$ is a bounded and connected domain with connected boundary $\partial D$, and $\partial D$ is locally the graph of a Lipschitz domain. More precisely, for any $Q\in \partial D$, there exists a circular coordinate cylinder $Z_{r}(Q)=\{X=(\underline{x}, x_{n})\in \mathbb{R}^{n}:|\underline{x}|<r, |x_{n}|<(1+L)r\}$ with $L,r>0$ which is centered at $Q$ (viz., $Q=0$ in this cylinder) and whose bases has positive distances from $\partial D$, and Lipschitz functions $\varphi_{Q}: \mathbb{R}^{n-1}\rightarrow \mathbb{R}$ such that
\begin{description}
  \item[(i)] $|\varphi_{Q}(\underline{x})-\varphi_{Q}(\underline{y})|\leq \mathcal{L}_{Q}|\underline{x}-\underline{y}|$ for any $\underline{x}, \underline{y}\in \mathbb{R}^{n-1}$ with $0<\mathcal{L}_{Q}<\infty$;
  \item[(ii)] $Z_{r}(Q)\cap D=\{(\underline{x}, x_{n}): x_{n}>\varphi_{Q}(\underline{x})\}$;
  \item[(iii)] $Z_{r}(Q)\cap \partial D=\{(\underline{x}, x_{n}): x_{n}=\varphi_{Q}(\underline{x})\}$;
  \item[(iv)] $Q=(\underline{0}, \varphi_{Q}(\underline{0}))$,
\end{description}
where $\underline{x}=\{x_{1}, \ldots, x_{n-1}\}\in \mathbb{R}^{n-1}$. By (\textbf{iv}) and $Q=0$, $\varphi_{Q}(\underline{0})=0$. This requirement always can be attained by some translations and rotations of the circular coordinate cylinders.  Let $\mathcal{L}=\sup_{Q\in \partial D}\mathcal{L}_{Q}$, by compactness, it is easy to see that $\mathcal{L}$ is finite. $\mathcal{L}$ is usually called the Lipschitz constant (or Lipschitz character) of $D$.

If $D=\{(\underline{x},x_{n})\in \mathbb{R}^{n}:\,\,x_{n}>\varphi(\underline{x})\}$ in which $\varphi:\mathbb{R}^{n-1}\rightarrow \mathbb{R}$ is Lipschitz continuous, then $D$ is called a Lipschitz graph domain. Sometimes some investigations of elliptic BVPs were considered on Lipschitz graph domains (see \cite{dk,pv1,du}).

Let $\Gamma_{\gamma}(Q)$ denote the non-tangentially cone-shaped region with vertex at $Q\in\partial D$ in $D$,

\begin{equation}
\Gamma_{\gamma}(Q)=\{X\in D: |X-Q|<\gamma\, \mathrm{dist}(X,\partial
D)\}
\end{equation}
in which $\gamma>1$. The non-tangential maximal functions of $F$ are defined as follows:
\begin{equation}
M(F)(Q)=\sup_{X\in \Gamma_{\gamma}(Q)}|F(X)|,\,\,Q\in
\partial D.
\end{equation}
The non-tangential limit of $F$ means that $\displaystyle\lim_{\substack{X\rightarrow P\\ X\in
\Gamma_{\gamma}(P), P\in\partial D}}\!\!\!\!\!F(X)$ exist. It is worthy to note that the non-tangential
maximal functions and limits
are defined for all $\gamma>1$ throughout this article, so we always elide the subscript
$\gamma$ in proper places and denote $\Gamma_{\gamma}(\cdot)$ only
by $\Gamma(\cdot)$. The boundary data in what follows are uniformly non-tangential.

\subsection{Single layer potential and harmonic Robin problem}
Denote the single layer potential of $f$ as follows
\begin{align}
\mathcal{S}f(X)&=\int_{\partial D}\Gamma(|X-Q|)f(Q)d\sigma(Q)\\
&=\frac{1}{\omega_{n-1}(2-n)}\int_{\partial D}\frac{1}{|X-Q|^{n-2}}f(Q)d\sigma(Q),\,\,\,X\in \mathbb{R}^{n},\nonumber
\end{align}
where $\Gamma(X)=\frac{1}{\omega_{n-1}(2-n)}|X|^{2-n}$ is the fundamental solution of Laplacian for $n\geq 3$,  $\omega_{n-1}=\frac{2\pi^{\frac{n}{2}}}{\Gamma(\frac{n}{2})}$ is the surface area of the unit sphere $S^{n-1}$ in $\mathbb{R}^{n}$, and $d\sigma$ is the surface measure of $\partial D$.

It is well known that the non-tangentially positive and negative boundary values of $\frac{\partial}{\partial N_{P}}\mathcal{S}f$ exist and can be formulated by
\begin{equation}
\left(\frac{\partial}{\partial N_{P}}\mathcal{S}\right)_{+}f(P)=\lim_{\substack{X\rightarrow P\\X\in \Gamma_{\gamma}(P)}}\frac{\partial}{\partial N_{P}}\mathcal{S}f(X)=-\frac{1}{2}f(P)+K^{*}f(P),\,\,P\in \partial D
\end{equation}
and
\begin{equation}
\left(\frac{\partial}{\partial N_{P}}\mathcal{S}\right)_{-}f(P)=\lim_{\substack{X\rightarrow P\\X\in \widetilde{\Gamma}_{\gamma}(P)}}\frac{\partial}{\partial N_{P}}\mathcal{S}f(X)=\frac{1}{2}f(P)+K^{*}f(P),\,\,P\in \partial D
\end{equation}
for any $f\in L^{p}(\partial D)$ with $1<p<\infty$, where $\frac{\partial}{\partial N_{P}}\mathcal{S}f=\langle\nabla \mathcal{S}f, N_{P}\rangle$, $\widetilde{\Gamma}_{\gamma}(P)=\{X\in \mathbb{R}^{n}\setminus D: |X-P|<\gamma\, \mathrm{dist}(X,\partial
D)\}$,
\begin{align}
K^{*}f(P)&=\mathrm{P.V.}\, \frac{1}{\omega_{n-1}}\int_{\partial D}\frac{\langle P-Q, N_{P} \rangle}{|P-Q|^{n}}f(Q)d\sigma(Q)\\
&=\lim_{\epsilon\rightarrow 0+}\frac{1}{\omega_{n-1}}\int_{\partial D\setminus\{|Q-P|<\epsilon\}}\frac{\langle P-Q, N_{P} \rangle}{|P-Q|^{n}}f(Q)d\sigma(Q),\nonumber
\end{align}
and $N_{P}$ denotes the unit outward normal of $\partial D$ at $P$.

By the deep result (i.e., the $L^{2}$ boundedness of Cauchy integrals on Lipschitz curves) due to Coifman, McIntosh and Meyer \cite{cmm}, the operator $K^{*}$ is bounded from $L^{p}(\partial D)$ to $L^{p}(\partial D)$ for any $1<p<\infty$. In order to apply the method of layer potentials (indeed, the boundary integral method) to solve Dirichlet and Neumann problems on Lipschitz domains, Kenig, Dahlberg and Verchota studied the invertibility of the traces of the classical layer potentials (including single and double layer potentials). More precisely, part of their results are stated in the following lemma.
\begin{lem}{(\cite{dk1} and \cite{v0})}
Let $D$ be a bounded Lipschitz domain (or Lipschitz graph domain), their exists an $\epsilon>0$ depending only on the dimension $n$ and the Lipschitz character $\mathcal{L}$ of $D$, such that the operators $\pm\frac{1}{2}I+K:L^{p}(\partial D)\rightarrow L^{p}(\partial D)$ are invertible for any $2-\epsilon<p<\infty$; the operators $\pm\frac{1}{2}I+K^{*}:L^{p}(\partial D)\rightarrow L^{p}(\partial D)$ are invertible for any $1<p<2+\epsilon$. Moreover, the single layer potential $\mathcal{S}:L^{p}(\partial D)\rightarrow L_{1}^{p}(\partial D)$ is invertible for any $1<p<2+\epsilon$, where $K$ is the adjoint operator of $K^{*}$.
\end{lem}

Suppose that the Robin coefficient $b$ satisfies the following conditions:\\
\begin{equation}b\geq0\,\,\, \mathrm{and}\,\,\,b\not\equiv0\,\,\mathrm{on} \,\,\partial D;  \end{equation}
\begin{equation}b\in L^{n-1}(\partial D)\,\,\mathrm{if}\,\,1<p<2\,\, \mathrm{and}\,\,n\geq4,\,\,\mathrm{or}\,\,1<p<2\,\,\mathrm{and}\,\,n=3; \end{equation}
\begin{equation}b\in L^{s}(\partial D)\,\,\mathrm{for}\,\,\mathrm{some}\,\,s>2\,\,\mathrm{if}\,\, p=2\,\,\mathrm{or}\,\, n=3.\end{equation}

Let $\mathcal{T}=-\frac{1}{2}I+K^{*}+b\mathcal{S}$, using Lemma 2.1, and some results of compact and Fredholm operators, Lanzani and Shen established the invertibility of $\mathcal{T}$ in \cite{ls} as follows
\begin{lem}{(\cite{ls})}
Suppose that $D$ is a bounded Lipschitz domain and the Robin coefficient $b$ satisfies the conditions (2.7)-(2.9), then $\mathcal{T}$ is invertible from $L^{p}(\partial D)$ onto $L^{p}(\partial D)$ for any $1<p\leq 2$.
\end{lem}

Moreover, they also obtained the following theorem of solvability of harmonic Robin problem in $L^{p}$ on Lipschitz domains as follows
\begin{thm}{(\cite{ls})}
Let $D$ be a bounded Lipschitz domain, the coefficient $b$ satisfy (2.7)-(2.9), then there exists a unique solution of the harmonic Robin problem (R) on $D$ with boundary data $h\in L^{p}(\partial D)$, which satisfies the estimate $\|M(\nabla u)\|_{L^{p}(\partial D)}\leq C\|h\|_{L^{p}(\partial D)}$, for any $1<p\leq 2$, and the solution can be represented as $u=\mathcal{S}(\mathcal{T}^{-1})h$.
\end{thm}

\begin{rem}
With respect to general boundary integral method, the key idea is to establish the invertibility of boundary integral operators under appropriate conditions. Therefore the crux of Verchota, Dahlberg and Kenig's results is to find the invertibility of trace operators of the classical single and double layer potentials under necessary conditions. So do Lanzani and Shen's results. This approach is relatively valid and esay for harmonic BVPs in most cases, while it is almost impossible and difficult for polyharmonic BVPs because the prescribed boundary data are multitudinous for the latter comparing with the former. Even for harmonic BVPs, it is difficult to solve by the method of layer potentials when the prescribed boundary data increase. One example is the following mixed problem (also called Zaremba problem):
\begin{equation*}{(Z)}\hspace{5mm}\begin{cases}\Delta u=0\,\,in\,\, D,\\
 u=f\,\,on\,\,\mathcal{D},\\
 \frac{\partial u}{\partial N}=g\,\,on\,\,\mathcal{N},
  \end{cases}
  \end{equation*}
where $D$ is a domain in $\mathbb{R}^{n}$ with boundary $\partial D$, $\mathcal{D}\cup\mathcal{N}=\partial D$, both $\mathcal{D}$ and $\mathcal{N}$ are nonempty, $f,g$ are given functions in some certain function spaces respectively defined on $\mathcal{D}$ and $\mathcal{N}$. Thus it must be needed to introduce some new ideas in order to solve polyharmonic (or general higher order elliptic) BVPs by using the method of layer potentials, which will include self-developments of the theory of layer potentials.
\end{rem}

\subsection{Polyharmonic fundamental solutions and multi-layer $\mathcal{S}$-potentials}

In \cite{du}, Du explicitly represents the fundamental solutions for polyharmonic operators $\Delta^{m}$, $2\leq m<\infty$, in a new and different form although they have been presented in many earlier literatures with very vague coefficients (for instance, see \cite{acl} and \cite{sb}). Then multi-layer $\mathcal{S}$-potentials (higher order analog of the single layer potential $\mathcal{S}$) are firstly introduced in terms of these fundamental solutions. More precisely,

Set \begin{equation}
\delta_{s}=s(s+n-2)
\end{equation}
for any $s\in \mathbb{R}\setminus\{0\}$.
For $X, Y\in \mathbb{R}^{n}$, denote
\begin{equation}
\mathcal{K}_{1}(X,Y)=\frac{1}{(n-2)\omega_{n-1}}\frac{1}{|X-Y|^{n-2}},
\end{equation}
For
$m\in \mathbb{N}$ and $m\geq 2$, define
\begin{align}
\mathcal{K}_{m}(x,v)=\frac{1}{(n-2)\omega_{n-1}\gamma_{1}\gamma_{2}\cdots\gamma_{m-1}}|X-Y|^{2m-n}
\end{align}
if $n$ is odd, and
\begin{align}
\mathcal{K}_{m}(x,v)=\begin{cases}\frac{1}{(n-2)\omega_{n-1}\gamma_{1}\gamma_{2}\cdots\gamma_{m-1}}|X-Y|^{2m-n},\,\,\,\,m\leq \frac{n-2}{2},\vspace{2mm}\\
\frac{1}{(n-2)^{2}\omega_{n-1}\gamma_{1}\gamma_{2}\cdots\gamma_{\frac{n-1}{2}-1}\delta_{2}\delta_{4}\cdots\delta_{2m-n-1}}|X-Y|^{2m-n}\vspace{1mm}\\
\times\left[\log|X-Y|+\frac{1}{n}-\sum_{t=1}^{m-\frac{n}{2}}\left(\frac{1}{2t}+\frac{1}{2t+n-2}\right)\right],\,\,\,m\geq\frac{n}{2}
\end{cases}
\end{align}
if $n$ is even, where
\begin{equation}\gamma_{k}=\delta_{2k-n+2},\,\,\,k=1,2,\ldots,m-1.\end{equation}
Then
\begin{equation}
\Delta \mathcal{K}_{1}(X,Y)=0\,\,\mbox{and}\,\,\Delta
\mathcal{K}_{m}(X,Y)=\mathcal{K}_{m-1}(X,Y), \,\,m\geq2.
\end{equation}
It is also obvious to that $\mathcal{K}_{1}(X,Y)=\Gamma(X-Y)$, $\Gamma$ is the classical harmonic fundamental solution. For this reason, $\mathcal{K}_{m}$ is called the $m$th order harmonic fundamental solution, or the fundamental solution of the polyharmonic operator $\Delta^{m}$, $2\leq m<\infty$.

Denote
\begin{equation}
\mathcal{M}_{j}f(X)=\int_{\partial D} \mathcal{K}_{j}(X,Q)
f(Q)d\sigma(Q), \,\,X\in D,
\end{equation}
where $1\leq j< \infty$, $\mathcal{K}_{j}$ is the $j$th order
polyharmonic fundamental solution, $d\sigma$ is the surface measure
on $\partial D$, and $f\in L^{p}(\partial D)$ for some suitable $p$.
$\mathcal{M}_{j}f$ is called the $j$th-layer $\mathcal{S}$-potential
of $f$. Obviously, $\mathcal{M}_{1}$ is just the classical
single layer potential $\mathcal{S}$.

Multi-layer $\mathcal{S}$-potentials $\mathcal{M}_{j}$ have some nice properties stated in the following lemma.
\begin{lem}{(\cite{du})}
Let $D$ be a bounded Lipschitz domain with connected boundary $\partial D$, then\\
\rm{(1)}   $\frac{\partial}{\partial N_{P}}\mathcal{M}_{j}f(X)$ has the non-tangential limit when $X\in D$ non-tangentially approaches to almost every point $P\in \partial D$ and any $j\geq 2$. More precisely,
\begin{equation}\lim_{\substack{X\rightarrow P\\ X\in \Gamma(P)}}\frac{\partial}{\partial N_{P}}\mathcal{M}_{j}f(X)=\lim_{\substack{X\rightarrow P\\ X\in \Gamma(P)}}\int_{\partial D}\langle K_{j}(X, Q), N_{P}\rangle f(Q)d\sigma(Q)
        =\mathrm{K}_{j}^{*}f(P)\end{equation}
        for $f\in L^{p}(\partial D)$, $1\leq p\leq\infty$,
        where \begin{equation}\mathrm{K}_{j}^{*}f(P)=\int_{\partial D}\langle K_{j}(P, Q), N_{P}\rangle f(Q)d\sigma(Q),\,\,P\in\partial D;\end{equation}
\rm{(2)}  $\Delta \mathcal{M}_{j}f(X)=\mathcal{M}_{j-1}f(X)$ as $j\geq 2$, and $\Delta \mathcal{M}_{1}f(X)=0$ for any $f\in L^{p}(\partial D)$ with $1\leq p\leq\infty$;\\
\rm{(3)}  $\mathrm{K}_{j}^{*}: L^{p}(\partial D)\rightarrow L^{p}(\partial D)$ is bounded for any $1\leq p\leq\infty$ as $j\geq 2$;\\
\rm{(4)}  $\mathcal{M}_{j}: L^{p}(\partial D)\rightarrow L^{p}(D)$ is bounded for any $1\leq p\leq\infty$ as $j\geq 1$;\\
\rm{(5)}  $\mathcal{M}_{j}: L^{p}(\partial D)\rightarrow L_{1}^{p}(D)$ is bounded for any $1\leq p\leq\infty$ as $j\geq 2$;\\
\rm{(6)}  $M\circ (\nabla\mathcal{M}_{j}): L^{p_{j}}(\partial D)\rightarrow L^{p_{j}}(\partial D)$ is bounded for any $j\geq 1$ with $p_{j}\in\begin{cases}
(1, \infty), \,\,j=1;\vspace{0.5mm}\\
[1, \infty], \,\,j\geq 2.
\end{cases}$
\end{lem}

\section{$L^{p}$ polyharmonic Robin problems on Lipschitz domains}

In this section, we will solve the following polyharmonic Robin problems with $L^{p}$ boundary data on Lipschitz domains:
\begin{equation}\hspace{5mm}\begin{cases}\Delta^{m} u=0\,\,in\,\, D,\\
 \frac{\partial }{\partial N}\Delta^{j}u+b_{j}\Delta^{j}u=h_{j}\,\,on\,\,\partial D,\\
M(\nabla u)\in L^{p}(\partial D)
  \end{cases}
  \end{equation}
with $\|M(\nabla u)\|_{L^{p}(\partial D)}\leq C\sum_{j=0}^{m-1}\|h_{j}\|_{L^{p}(\partial D)}$, where $D$ is a bounded Lipschitz domain, $\Delta$ is the Laplacian, $h_{j}\in L^{p}(\partial D)$  with some $p\in (1,\infty)$, the Robin coefficients $b_{j}$ satisfies some appropriate conditions for any $0\leq j<m$, the constant $C$ in the estimate depends only on $m,n,p$ and the Lipschitz character of $D$.

Denote Robin boundary operators $R_{b_{j}}=\frac{\partial }{\partial N_{P}}+b_{j}(P)=\left<\nabla, N_{P}\right>+b_{j}(P)$, $0\leq j<m$, then more precisely, we have
\begin{thm}
Let $D$ be a bounded Lipschitz domain, $h_{j}\in L^{p}(\partial D)$ with $1<p\leq 2$, and the Robin coefficients $b_{j}$ satisfy the conditions (2.7)-(2.9) for any $0\leq j<m$, then there exists a solution for the polyharmonic Robin problem (3.1), and the solution can be represented as
\begin{align}
u(X)&=\sum_{j=1}^{m}\int_{\partial D}\mathcal{K}_{j}(X,
Q)\widetilde{h}_{j-1}(Q)d\sigma(Q),\\
&=\sum_{j=1}^{m}\mathcal{M}_{j}\widetilde{h}_{j-1}(X),\,\, X\in D,\nonumber
\end{align}
where
\begin{equation}\widetilde{h}_{m-1}=\mathcal{T}_{m-1}^{-1}h_{m-1}\end{equation}
and
\begin{equation}\widetilde{h}_{l}=\mathcal{T}_{l}^{-1}\left(h_{l}-\sum_{j=l+2}^{m}\mathrm{T}_{l,j-l}\widetilde{h}_{j-1}\right)\end{equation}
with $0\leq l\leq m-2$, which satisfies the following estimates
\begin{equation}
\|M(\nabla u)\|_{L^{p}(\partial D)}\leq C\sum_{j=0}^{m-1}\|h_{j}\|_{L^{p}(\partial D)},
\end{equation}
\begin{equation}
\|u\|_{L^{p}(D)}\leq C\sum_{j=0}^{m-1}\|h_{j}\|_{L^{p}(\partial D)}
\end{equation}
and
\begin{equation}
\|\nabla(u-\mathcal{M}_{1}\widetilde{h}_{0})\|_{L^{p}(D)}\leq C\sum_{j=1}^{m-1}\|h_{j}\|_{L^{p}(\partial D)}
\end{equation}
in which $M(\nabla u)$ is the non-tangential maximal function of $\nabla u$ on $\partial D$, $\mathcal{T}_{l}=-\frac{1}{2}I+K^{*}+b_{l}\mathcal{S}$, and $\mathrm{T}_{l,j-l}=(R_{b_{l}}\circ \mathcal{M}_{j-l})\upharpoonright_{\partial D}=K^{*}_{j-l}+b_{l}\mathcal{M}_{j-l}$, $0\leq l<m$, $l+2\leq j\leq m$.
The solution (3.2) with (3.3) and (3.4) is unique under (3.6), and unique up to a constant under (3.5) or (3.7).
\end{thm}

\begin{proof}
For the existence, denote the solution in form as
\begin{equation}
u(X)=\mathcal{M}_{1}\widetilde{h}_{0}(X)+\mathcal{M}_{2}\widetilde{h}_{1}(X)+\cdots+\mathcal{M}_{m}\widetilde{h}_{m-1}(X)
\end{equation}
for some functions $\widetilde{h}_{j}$, $0\leq j\leq m-1$ to be given below, where $\mathcal{M}_{j}$ is the $j$th-layer $\mathcal{S}$-potential.

Taking the polyharmonic operators $\Delta^{l}$, $0\leq
l\leq m$, on two sides of (3.8), in form, we have
\begin{align}
\begin{cases}
u(X)&=\mathcal{M}_{1}\widetilde{h}_{0}(X)+\mathcal{M}_{2}\widetilde{h}_{1}(X)+\mathcal{M}_{3}\widetilde{h}_{2}(X)+\cdots+\mathcal{M}_{m}\widetilde{h}_{m-1}(X),\vspace{1mm}\\
\Delta u(X)&=\mathcal{M}_{1}\widetilde{h}_{1}(X)+\mathcal{M}_{2}\widetilde{h}_{2}(X)+\cdots+\mathcal{M}_{m-1}\widetilde{h}_{m-1}(X),\nonumber\vspace{1mm}\\
\Delta^{2} u(X)&=\mathcal{M}_{1}\widetilde{h}_{2}(X)+\cdots+\mathcal{M}_{m-2}\widetilde{h}_{m-1}(X),\nonumber\vspace{1mm}\\
&\cdots\nonumber\vspace{1mm}\\
\Delta^{m-1} u(X)&=\mathcal{M}_{1}\widetilde{h}_{m-1}(X),\nonumber\vspace{1mm}\\
\Delta^{m} u(X)&=0.
\end{cases}
\end{align}
Thus for almost everywhere $P\in \partial D$ and $X\in\Gamma(P)$, let the Robin differential operators $R_{b_{j}}$, $0\leq j<m$, acting respectively on two sides of the former $m$ identities in the above system, viz.,
\begin{align}
\begin{cases}
\left(\frac{\partial }{\partial N_{P}}+b_{0}(P)\right)u(X)&=\left(\frac{\partial }{\partial N_{P}}+b_{0}(P)\right)\left[\mathcal{M}_{1}\widetilde{h}_{0}(X)+\cdots+\mathcal{M}_{m}\widetilde{h}_{m-1}(X)\right],\vspace{1mm}\\
\left(\frac{\partial }{\partial N_{P}}+b_{1}(P)\right)\Delta u(X)&=\left(\frac{\partial }{\partial N_{P}}+b_{1}(P)\right)\left[\mathcal{M}_{1}\widetilde{h}_{1}(X)+\cdots+\mathcal{M}_{m-1}\widetilde{h}_{m-1}(X)\right],\nonumber\vspace{1mm}\\
\left(\frac{\partial }{\partial N_{P}}+b_{2}(P)\right)\Delta^{2} u(X)&=\left(\frac{\partial }{\partial N_{P}}+b_{2}(P)\right)\left[\mathcal{M}_{1}\widetilde{h}_{2}(X)+\cdots+\mathcal{M}_{m-2}\widetilde{h}_{m-1}(X)\right],\nonumber\vspace{1mm}\\
&\cdots\nonumber\vspace{1mm}\\
\left(\frac{\partial }{\partial N_{P}}+b_{m-1}(P)\right)\Delta^{m-1} u(X)&=\left(\frac{\partial }{\partial N_{P}}+b_{m-1}(P)\right)\mathcal{M}_{1}\widetilde{h}_{m-1}(X).\nonumber
\end{cases}
\end{align}
Taking $X\in D$ converge to $P\in\partial D$ non-tangentially, together with the boundary data of (3.1), then
\begin{align}
\begin{cases}
h_{0}(P)&=\mathcal{T}_{0}\widetilde{h}_{0}(P)+\mathrm{T}_{0,2}\widetilde{h}_{1}(P)+\mathrm{T}_{0,3}\widetilde{h}_{2}(P)+\cdots+\mathrm{T}_{0,m}\widetilde{h}_{m-1}(P),\vspace{1mm}\\
h_{1}(P)&=\mathcal{T}_{1}\widetilde{h}_{1}(P)+\mathrm{T}_{1,2}\widetilde{h}_{2}(P)+\cdots+\mathrm{T}_{1,m-1}\widetilde{h}_{m-1}(P),\nonumber\vspace{1mm}\\
h_{2}(P)&=\mathcal{T}_{2}\widetilde{h}_{2}(P)+\cdots+\mathrm{T}_{m-1,m-2}\widetilde{h}_{m-1}(P),\nonumber\vspace{1mm}\\
&\cdots\nonumber\vspace{1mm}\\
h_{m-1}(P)&=\mathcal{T}_{m-1}\widetilde{h}_{m-1}(P).\nonumber
\end{cases}
\end{align}

By the invertible property of $\mathcal{T}$ and $L^{p}$ boundness of $\mathrm{T}_{l,j-l}$, $0\leq l<m$, $l+2\leq j\leq m$, then
\begin{align}
\begin{cases}
\widetilde{h}_{0}(P)&=\mathcal{T}_{0}^{-1}\left[h_{0}(P)-\mathrm{T}_{0,2}\widetilde{h}_{1}(P)-\mathrm{T}_{0,3}\widetilde{h}_{2}(P)-\cdots-\mathrm{T}_{0,m}\widetilde{h}_{m-1}(P)\right],\vspace{1mm}\\
\widetilde{h}_{1}(P)&=\mathcal{T}_{1}^{-1}\left[h_{1}(P)-\mathrm{T}_{1,2}\widetilde{h}_{2}(P)-\cdots-\mathrm{T}_{1,m-1}\widetilde{h}_{m-1}(P)\right],\vspace{1mm}\nonumber\\
\widetilde{h}_{2}(P)&=\mathcal{T}_{2}^{-1}\left[h_{2}(P)-\cdots-\mathrm{T}_{2,m-2}\widetilde{h}_{m-1}(P)\right],\vspace{1mm}\nonumber\\
&\cdots\nonumber\\
\widetilde{h}_{m-1}(P)&=\mathcal{T}_{m-1}^{-1}h_{m-1}(P).\nonumber
\end{cases}
\end{align}

Therefore, we obtain
\begin{align}
\begin{cases}
\widetilde{h}_{m-1}=\mathcal{T}_{m-1}^{-1}h_{m-1},\vspace{2mm}\\
\widetilde{h}_{l}=\mathcal{T}_{l}^{-1}\left[h_{l}-\sum_{j=l+2}^{m}\mathrm{T}_{l,j-l}\widetilde{h}_{j-1}\right],
\end{cases}
\end{align}
where $0\leq l\leq m-2$.
Uniformly write
\begin{equation}
\widetilde{h}_{l}=\mathcal{T}_{l}^{-1}\left(h_{l}-\sum_{j=l+2}^{m}\mathrm{T}_{l,j-l}\widetilde{h}_{j-1}\right)
\end{equation}
with $0\leq l\leq m-1$ in view of the convention that $\sum_{j=l}^{k}s_{j}=0$ as $k<l$.

By Lemmas 2.2 and 2.4, it is easy to see that the formal reasoning in the above makes sense whenever $h_{j}\in L^{p}(\partial D)$, $1\leq j\leq m-1$, $1<p\leq 2$. That is, a solution of (3.1) is (3.2) with (3.3) and (3.4).

Next we turn to the estimate and uniqueness of the solution. By Lemmas 2.2, 2.3 and 2.5, we have
\begin{align}\|M(\nabla u)\|_{L^{p}(\partial D)}&=\|M\left(\sum_{j=1}^{m}\nabla \mathcal{M}_{j}\widetilde{h}_{j-1}\right)\|_{L^{p}(\partial D)}\\
&\leq \sum_{j=1}^{m}\|M(\nabla \mathcal{M}_{j}\widetilde{h}_{j-1}) \|_{L^{p}(\partial D)}\nonumber\\
&\leq C\sum_{j=0}^{m-1}\| \widetilde{h}_{j} \|_{L^{p}(\partial D)}\nonumber\\
&= C\sum_{j=0}^{m-1}\| \mathcal{T}_{j}^{-1}\left(h_{j}-\sum_{s=j+2}^{m}\mathrm{T}_{j,s-j}\widetilde{h}_{s-1}\right) \|_{L^{p}(\partial D)}\nonumber\\
&\leq \widetilde{C}\sum_{j=0}^{m-1}\| h_{j} \|_{L^{p}(\partial D)}\nonumber
\end{align}
where $1<p\leq 2$, and the constants $C$ and $\widetilde{C}$ depends only on $m,n,p$ and $D$.

Similarly, by Lemmas 2.2, 2,3 and 2.5, we can also obtain the estimates (3.6) and (3.7). It is easy to get that the solution (3.2) is unique under (3.6), and unique up to a constant under (3.5) or (3.7). The proof is complete.
\end{proof}

\section{The $C^{1}$ domains case}

When Lipschitz domains $D$ are replaced by $C^{1}$ domains, the polyharmonic Robin problems (3.1) can be also solved with boundary data in $L^{p}(\partial D)$ for any $1<p<\infty$.

\begin{thm}
Let $D$ be a bounded $C^{1}$ domain. For any $0\leq j<m$,  $h_{j}\in L^{p}(\partial D)$ with $1<p<\infty$, the Robin coefficients $b_{j}$ satisfy the conditions (2.7) and
\begin{equation}
b_{j}\in L^{s}(\partial D)\,\,\,\mathrm{with}\,\,\,s=\begin{cases}n-1,\,\,\,\,\mathrm{when}\,\, 1<p<n-1;\\
n-1+\epsilon,\,\,\,\,\mathrm{when}\,\, p=n-1;\\
p,\,\,\,\,\,\mathrm{when}\,\,p>n-1\end{cases}
\end{equation}
where $\epsilon >0$, then there exists a solution for the polyharmonic Robin problem (3.1), and the solution can be represented as
\begin{align}
u(X)&=\sum_{j=1}^{m}\int_{\partial D}\mathcal{K}_{j}(X,
Q)\widetilde{h}_{j-1}(Q)d\sigma(Q),\\
&=\sum_{j=1}^{m}\mathcal{M}_{j}\widetilde{h}_{j-1}(X),\,\, X\in D,\nonumber
\end{align}
where
\begin{equation}\widetilde{h}_{m-1}=\mathcal{T}_{m-1}^{-1}h_{m-1}\end{equation}
and
\begin{equation}\widetilde{h}_{l}=\mathcal{T}_{l}^{-1}\left(h_{l}-\sum_{j=l+2}^{m}\mathrm{T}_{l,j-l}\widetilde{h}_{j-1}\right)\end{equation}
with $0\leq l\leq m-2$, which satisfies the following estimates
\begin{equation}
\|M(\nabla u)\|_{L^{p}(\partial D)}\leq C\sum_{j=0}^{m-1}\|h_{j}\|_{L^{p}(\partial D)},
\end{equation}
\begin{equation}
\|u\|_{L^{p}(D)}\leq C\sum_{j=0}^{m-1}\|h_{j}\|_{L^{p}(\partial D)}
\end{equation}
and
\begin{equation}
\|\nabla(u-\mathcal{M}_{1}\widetilde{h}_{0})\|_{L^{p}(D)}\leq C\sum_{j=1}^{m-1}\|h_{j}\|_{L^{p}(\partial D)}
\end{equation}
in which $M(\nabla u)$ is the non-tangential maximal function of $\nabla u$ on $\partial D$, $\mathcal{T}_{l}=-\frac{1}{2}I+K^{*}+b_{l}\mathcal{S}$, and $\mathrm{T}_{l,j-l}=(R_{b_{l}}\circ \mathcal{M}_{j-l})\upharpoonright_{\partial D}=K^{*}_{j-l}+b_{l}\mathcal{M}_{j-l}$, $0\leq l<m$, $l+2\leq j\leq m$.
The solution (4.2) with (4.3) and (4.4) is unique under (4.6), and unique up to a constant under (4.5) or (4.7).
\end{thm}

\begin{proof}
It is easy to show that the operators $\mathcal{T}_{j}: L^{p}(\partial D)\rightarrow L^{p}(\partial D)$, $0\leq j\leq m-1$, are invertible when the Robin coefficient $b_{j}$ satisfies the conditions (2.7) and (4.1) by a similar argument stated in \cite{ls}. Noting this fact, the theorem is established following from a argument similar to Theorem 3.1.
\end{proof}


\end{document}